\documentclass[12pt,reqno]{amsart}

\usepackage{amscd}
\usepackage{latexsym,graphics,enumerate}
\usepackage{mathrsfs}
\usepackage{amsfonts}
\usepackage{dsfont}
\usepackage{enumerate}
\usepackage{setspace}
\usepackage{color}
\usepackage{nicefrac}
\usepackage[normalem]{ulem}
\usepackage[page,header]{appendix}
\usepackage{titletoc}
\usepackage[mathscr]{eucal}
\usepackage{mathrsfs}

\usepackage{graphicx}
\usepackage{amsfonts}
\usepackage{amsmath}
\usepackage{amssymb}
\usepackage{amsthm}
\usepackage{indentfirst}
\usepackage{booktabs}
\usepackage{verbatim}
\usepackage{color}

\usepackage{subfigure}

\usepackage[left=1in,right=1in,top=1in,bottom=1in]{geometry}

\usepackage[colorlinks=true,
            linkcolor=red,
            urlcolor=blue,
            citecolor=cyan]{hyperref}

\newcommand{\be}{\begin{equation}} 
\newcommand{\bel}{\begin{eqnarray}}
\newcommand{\ee}{\end{equation}} 
\newcommand{\eel}{\end{eqnarray}}
\newcommand{\ba}{\begin{aligned}}
\newcommand{\ea}{\end{aligned}}

\newcommand{\rev}[1]{#1} 

\numberwithin{equation}{section}

\newtheorem{theorem}{Theorem}[section]
\newtheorem{lemma}{Lemma}[section]

\newtheorem{proposition}[theorem]{Proposition}

\newtheorem{remark}{Remark}[section]

\renewcommand{\geq}{\geqslant}
\renewcommand{\ge}{\geqslant}
\renewcommand{\leq}{\leqslant}
\renewcommand{\le}{\leqslant}
\newcommand{\rd}{\,\mathrm{d}}
\newcommand{\dx}{\rd{\bx}}
\newcommand{\dy}{\rd{\by}}

\def\bu{{\bf u}}
\def\bx{{\bf x}}
\def\by{{\bf y}}

\def\cW{\mathscr{W}}

\def\rW{{W}}
\def\cw{{w}}
\def\rw{{\sf{w}}}
\def\cN{\mathscr{N}}
\def\rN{{N}}
\def\cV{\mathscr{V}}
\def\rV{{V}}
\def\cD{\mathscr{D}}

\def\tV{\widetilde{\cV}}

\def\supp{\textnormal{supp\,}}
\def\ddt{\frac{\rd}{\rd{t}}}

\def\BR{{\mathbb B}_R}
\def\BRzero{{\mathbb B}_{R_0}}
\def\B1{{\mathbb B}_1}
\def\BRinf{{\mathbb B}_{R_\infty}}

\def\argmin{\textnormal{argmin}}

\DeclareRobustCommand{\rchi}{{\mathpalette\irchi\relax}}
\newcommand{\irchi}[2]{\raisebox{\depth}{$#1\chi$}} 

\def\hf{\frac{1}{2}}

\begin{document}

\title[Newtonian repulsion and radial confinement]
{Newtonian repulsion and radial confinement: convergence towards steady state}

\author{Ruiwen Shu}
\address{Department of Mathematics, Center for Scientific Computation and Mathematical Modeling (CSCAMM)\newline
University of Maryland, College Park MD 20742}
\email{rshu@cscamm.umd.edu}

\author{Eitan Tadmor}
\address{Department of Mathematics, Center for Scientific Computation and Mathematical Modeling (CSCAMM), and Institute for Physical Sciences \& Technology (IPST)\newline
University of Maryland, College Park MD 20742}
\email{tadmor@umd.edu}



\thanks{\textbf{Acknowledgment.} Research was supported in part by NSF and ONR grants DMS16-13911and N00014-1812465.}
\date{February 11, 2021}

\begin{abstract}
We investigate the large time behavior of multi-dimensional aggregation equations driven by Newtonian repulsion, and balanced by radial attraction and confinement.
In case of  Newton repulsion with radial confinement we  quantify the algebraic convergence decay rate towards the unique steady  state. To this end, we identify a one-parameter family of radial steady states, and  prove dimension-dependent  decay rate in energy and 2-Wassertein distance, using a comparison with properly selected radial steady states.  We also study Newtonian repulsion and  radial attraction.  When the attraction potential is quadratic it is known to coincide with  quadratic confinement. Here we 
study the case of perturbed radial quadratic attraction, proving  that it still leads to one-parameter family of unique steady states. It is expected that this family to serve for a corresponding comparison argument which yields   algebraic convergence towards  steady repulsive-attractive solutions. 
\end{abstract}

\maketitle
\tableofcontents

\section{Introduction}

In this paper we study the large time behavior of the  first-order aggregation equation
\begin{equation}\label{eq0}
\partial_t \rho + \nabla\cdot(\rho \bu) = 0, \qquad \bu(t,\bx)=-\nabla\Phi(t,\bx),
\end{equation}
subject to prescribed initial distribution, $\rho(0,\bx)=\rho_0(\bx)$, with  mass  
\begin{equation}
m_0 = \int \rho_0(\bx)\dx = \int \rho(t,\bx)\dx > 0, \quad \forall t>0.
\end{equation}
 
The dynamics we have in mind for \eqref{eq0} governs the interaction of infinitesimal mass elements, $\rho(t,\bx)\dx$,   which are dominated by repulsion near in the immediate neighborhood of $\bx\in {\mathbb R}^d$ and  balanced by attraction and confinement  which dominate away from $\bx$. This reflects     ``social'' interactions   encountered in applications --- describing collective dynamics in ecology, human interactions or sensor-based crowds, \cite{CMV03,CMV06,FHK11,KSUB11,BCLR13,BCY14,CFT15, CFP17}, ... .
 In this paper, we consider the case of Newtonian repulsion $\nabla (-\Delta)^{-1}\rho(t,\bx)$ coupled with attraction $\nabla \cW*\rho(t,\bx)$  and confinement $\nabla \cV(\bx)$,
\begin{equation}\label{whatisPhi}
\bu(t,\bx) = -\nabla \Phi(t,\bx), \quad \Phi(t,\bx):=\int \cN(\bx-\by)\rho(t,\by)\dy +\int  \cW(\bx-\by)\rho(t,\by)\dy + \cV(\bx).
\end{equation}
Here, $\rho(t,\bx)\ge 0$ is the large crowd density distribution  of ``agents'', varying in time-space $(t,\bx) \in(\mathbb{R}_{+}\times\mathbb{R}^d)$, $\cN$ is the Newtonian potential satisfying $\Delta \cN = -\delta$, 
\begin{equation}
\cN(\bx) = \left\{\begin{split}
& - \frac{1}{2}|\bx|,\qquad \qquad d=1 \\
& -\frac{1}{2\pi}\log|\bx|,\ \ \quad d=2 \\
& \frac{c_d}{|\bx|^{d-2}},\qquad \qquad \ c_d>0, d\ge 3 \\
\end{split}\right.
\end{equation}
and $\cV(\bx)=\rV(r)$ and $\cW(\bx) = \rW(r),\,r = |\bx|,\,\rev{\rW'(r)\ge 0}$ are confining external potential and, respectively, a pairwise attraction potential, both are assumed radial, \rev{smooth} and with Pareto tail at infinity 
\begin{equation}\label{eq:pareto}
\lim_{r\rightarrow \infty}\rV'(r) r^{d-1} =\infty,
\end{equation}
so that the external potential  (--- and likewise, the pairwise interaction potential) dominates the Newtonian repulsion at infinity, $\lim_{R\rightarrow \infty}\rV(R)/\rN(R) =\infty$.

This paper is concerned with the large time behavior of the aggregation equation \eqref{eq0}, when  Newtonian repulsion is balanced by the presence of either confinement or attraction induced by  a potential, $\cV$, or respectability, $\cW$. Observe that a steady state  of \eqref{eq0}, $\rho_\infty$, is characterized\footnote{A steady solution of \eqref{eq0}, $\nabla\cdot(\rho_\infty\nabla\Phi_\infty)=0$, implies $\displaystyle \int \rho_\infty|\nabla\Phi_\infty|^2\dx=0$, i.e., $\bu_\infty$ vanishes on $\supp\rho_\infty$ in agreement with \eqref{dE} below.} by a  velocity field which vanishes on the support of $\rho$, i.e.,
\begin{equation}\label{rhoinfty}
-\int \nabla \cN(\bx-\by)\rho_\infty(\by)\rd{\by} -\int \nabla \cW(\bx-\by)\rho_\infty(\by)\rd{\by} -\nabla \cV(\bx) = 0,\quad \forall \bx\in\supp\rho_\infty.
\end{equation}
Taking divergence, then \eqref{rhoinfty} implies
\begin{equation}\label{rhosuppeq}
\rho_\infty(\bx) = \int \Delta \cW(\bx-\by)\rho_\infty(\by)\rd{\by} +\Delta \cV(\bx),\quad \forall \bx\in\supp\rho_\infty,
\end{equation}
which appears to be a key property of steady states.
The set of steady states is not empty: indeed, \eqref{eq0} is the 2-Wasserstein gradient flow of the total energy
\[
E[\rho] = \frac{1}{2}\iint \cN(\bx-\by) \rho(\by)\rho(\bx)\rd{\by}\rd{\bx} + \frac{1}{2}\iint \cW(\bx-\by) \rho(\by)\rho(\bx)\rd{\by}\rd{\bx} + \int \cV(\bx)\rho(\bx)\rd{\bx},
\]
i.e., its solution $\rho(t,\bx)$ satisfies the energy dissipation law
\begin{equation}\label{dE}
\ddt E(t) = -\int |\bu(t,\bx)|^2\rho(t,\bx)\rd{\bx}:=-\cD[\rho(t,\cdot)],\quad E(t) := E[\rho(t,\cdot)].
\end{equation}
By compactness arguments $E[\rho]$ admits a global energy minimizer,  
$\{\rho_\infty: \ \cD[\rho_\infty]=0\}$, which is a steady state of \eqref{eq}. The main question,  therefore, is whether the steady state $\rho_\infty$ is unique, and whether the solution $\rho(t,\cdot)$ converges to $\rho_\infty$ as $t\rightarrow\infty$.

\section{Main results}

We will use $C$ and $c$ to denote positive constants, being large and small respectively, which may depend on $\cV$, $\cW$, and $\rho_0$, but otherwise, are independent of the other parameters; their specific values may change from one equation to the next.
For notation simplicity, we will assume $d\ge 2$ in the rest of this paper. The counterparts of all results for $d=1$ are rather straightforward, and outlined in the Appendix. $\BR$ denotes the $d$-dimensional ball $\BR=\{\bx\,:\, |\bx|\leq R\}$.

\subsection{Newtonian repulsion with external confining potential}

We first present the results for \eqref{eq0} with $\cW=0$, i.e., the model with Newtonian repulsion and external confining potential
\begin{equation}\label{eq}
\partial_t \rho + \nabla\cdot(\rho \bu) = 0,\quad \bu(t,\bx) = -\int \nabla \cN(\bx-\by)\rho(t,\by)\rd{\by}  -\nabla \cV(\bx).
\end{equation}
The repulsion-confinement \rev{equation} \eqref{eq} is the gradient flow of the \rev{corresponding} energy dissipation law
\begin{equation}\label{eq:eqEV}
E[\rho] = \frac{1}{2}\iint \cN(\bx-\by) \rho(\by)\rho(\bx)\rd{\by}\rd{\bx} +  \int \cV(\bx)\rho(\bx)\rd{\bx}.
\end{equation}
We note in passing that  at least formally, \eqref{eq0} is a  2-Wasserstein gradient flow of the total energy $E[\rho]$; consult \cite{CMV03,CMV06,CDFLS11,Cra17} for a rigorous derivation.

\medskip\noindent
{\bf Existence of global minimizer}. 
\rev{Our first result, summarized in Theorem \ref{thm_exist} below, proves the existence of 
\emph{compactly supported, global} energy minimizer of the repulsion-confinement energy functional \eqref{eq:eqEV}.
}
\begin{theorem}\label{thm_exist}
\rev{Consider the $d$-dimensional energy \eqref{eq:eqEV}, $d\ge 3$, with  $C^2$-potential $\cV$ such that $\displaystyle \lim_{|\bx|\rightarrow\infty} \cV(\bx) = \infty$. Given arbitrary $m_0>0$, it admits  a compactly supported global minimizer $\displaystyle \rho_\infty=\mathop{\argmin}_{\rho\in S} E(\rho)$ in
$\displaystyle S:=\big\{\rho\in L^1: \rho\ge 0,\,\int \rho\rd{\bx}=m_0\big\}$.}

\end{theorem}
Existence of minimizers for  energy functionals involving potentials
with a finite limit $\displaystyle \lim_{|\bx|\rightarrow\infty} \cV(\bx) = \cV_\infty$ goes back to P. L. Lions' original work on concentration-cancellation \cite[II.4]{Lio84}.
Related works on existence of minimizers for attraction-repulsion energy functionals using concentration compactness arguments can be found in \cite{SST15,CFT15,CCP15,BCT18} and using symmetry and symmetric rearrangement arguments in \cite{BG04,Lop19,FL19}.
Here we use compactness arguments to prove the existence of compactly supported global minimizer
in the admissible class  
\begin{equation}
S^M := \left\{\rho\in L^1\cap L^\infty: \rho\ge 0,\,\int \rho\rd{\bx}=m_0,\,\|\rho\|_{L^\infty} \le M\right\}.
\end{equation}
We first prepare the following comparison principle.
\begin{lemma}
Fix the constants $m_0, E_m$ and $\displaystyle M\ge \frac{2m_0}{|\B1|}$. For every $\rho \in S^M$  such that $E[\rho]\le E_m$, there exists $\rho_1\in S^M$  such that $E[\rho_1]\le E[\rho]$, with compact support $\supp\rho_1\subset \BR$,  for $R$ depending on $E_m$, $m_0$, $d$ and $\cV$ but independent of $M$.
\end{lemma}
\begin{proof}
Without loss of generality, assume $\min \cV = 0$. Let $R\ge 1$ be a large constant to be chosen. 

Set $\displaystyle \epsilon := \int_{\BR^c} \rho\rd{\bx}$. Since $\cN\ge 0$ for $d\ge 3$,
\begin{equation}
E[\rho\chi_{{}_{\BR}}] \le E[\rho] - \epsilon \min_{\bx\in \BR^c} \cV(\bx)
\end{equation}
We first choose $R$ large enough such that $\displaystyle \min_{\bx\in \BR^c} \cV(\bx) > \frac{2E_m}{m_0}$. This implies $\epsilon\le \hf m_0$ since $E[\rho\chi_{{}_{\BR}}]\ge 0$.
We now define $\rho_1$ 
\begin{equation}
\rho_1 :=\frac{m_0-2\epsilon}{m_0-\epsilon}\rho\chi_{{}_{\BR}} + \frac{2\epsilon}{|\B1|} \chi_{{}_{\B1}}.
\end{equation}
Then $\rho_1\ge 0, \ \displaystyle \int \rho_1\rd{\bx}=m_0$, and 
\[
\|\rho_1\|_{L^\infty} \le \frac{m_0-2\epsilon}{m_0-\epsilon}M + \frac{2\epsilon}{|\B1|}\le (1-\frac{\epsilon}{m_0})M + \frac{2\epsilon}{|\B1|} \le M.
\]
Thus, $\rho_1\in S^M$. Moreover, $\rho_1$ is compactly supported, $\supp\rho_1\subset \BR$, and it decreases the energy of $\rho$, for large enough $R$:
\[
\begin{split}
E[\rho_1] \le & E\left[\frac{m_0-2\epsilon}{m_0-\epsilon}\rho\chi_{{}_{\BR}}\right] + E\left[\frac{2\epsilon}{|\B1|} \chi_{{}_{\B1}}\right] \\
  & + \iint \cN(\bx-\by) \frac{2\epsilon}{|\B1|} \chi_{{}_{\B1}}(\by)\rd{\by}\frac{m_0-2\epsilon}{m_0-\epsilon}\rho\chi_{{}_{\BR}}(\bx)\rd{\bx} \\
\le & E[\rho\chi_{{}_{\BR}}] + C\epsilon + C\epsilon \| \cN*\chi_{{}_{\B1}}\|_{L^\infty} \\
\le & E[\rho] - \epsilon \min_{\bx\in \BR^c} \cV(\bx) + C\epsilon\\
\end{split}
\]
with $C$ depending on $m_0$ and $\max_{\bx\in \B1}\cV(\bx)$. Hence choosing $R$ large enough such that $\displaystyle \min_{\bx\in \BR^c} \cV(\bx) > C$,  then $E[\rho_1]\leq E[\rho]$ and the lemma follows.
\end{proof}

\begin{proof}[Proof of theorem \ref{thm_exist}]
Without loss of generality, assume $\min \cV = 0$. Therefore $E[\rho]\ge 0$ always holds. 
Fix any $\displaystyle M\ge \frac{2m_0}{|\B1|}$ and take an energy minimizing sequence $\{\rho^M_n\}$ in $S^M$, with $E[\rho_n] \le E_m$ where $E_m$ is a constant independent of $M$. By the previous lemma, we may replace $\rho^M_n$ by $\rho^M_{1,n}\in S^M$ such that $\supp\rho^M_{1,n}\subset \BR$ for some $R$ independent of $M$, and $\{\rho^M_{1,n}\}$ is still a minimizing sequence. Since $\{\rho^M_{1,n}\}$ is uniformly bounded in $L^p$ for any $1<p<\infty$, there exists a sub-sequence (still denoted as $\{\rho^M_{1,n}\}$) which converges weakly in $L^p$, to a limit denoted as $\rho_{\infty}^M$ with $\supp\rho_\infty^M\subset \BR$.  Weak $L^p$-convergence implies
\[
\lim_{n\rightarrow\infty}\int \rho^M_{1,n}\cV\rd{\bx} = \int \rho_\infty^M \cV\rd{\bx},
\]
and  since $(-\Delta)^{-1}\rho^M_{1,n}\in W^{2,p}$ converges $L^p$-strongly to $(-\Delta)^{-1}\rho_\infty^M$, then also 
\[
\lim_{n\rightarrow\infty}\int \rho^M_{1,n}(-\Delta)^{-1}\rho^M_{1,n}\rd{\bx}
= \int \rho_\infty^M (-\Delta)^{-1}\rho_\infty^M\rd{\bx}.
\]
We conclude that
\begin{equation}
\lim_{n\rightarrow\infty}E[ \rho^M_{1,n}] = E[\rho_\infty^M]
\end{equation}
i.e., $\rho_\infty^M$ is a global minimizer in $S^M$.

\noindent
We claim that
\begin{equation}\label{eq:rhoDelbound}
\|\rho_\infty^M\|_{L^\infty} \le M_\cV, \qquad M_\cV:=\|\Delta \cV\|_{L^\infty(\BR)},
\end{equation}
for any admissible $M$ (recalling that $R$ is independent of $M$). Indeed, if we assume that \eqref{eq:rhoDelbound} fails, then $\rho_\infty^M$ cannot be a steady state of \eqref{eq0} since  a steady state, by \eqref{rhosuppeq}, should satisfy $\rho_\infty(\bx) = \Delta \cV(\bx)\rchi_{\supp\rho_\infty}(\bx)$
and therefore cannot exceed $M_\cV$.
Let $\rho(t,\cdot)$ denote the solution to \eqref{eq0} subject to the ``non-steady'' initial condition $\rho_\infty^M$. The  evolving solution satisfies    $\|\rho(t,\cdot)\|_{L^\infty}\le \|\rho_\infty^M\|_{L^\infty}$ (c.f. STEP 1 of the proof of Theorem \ref{thm_Vequi} below), that is, $\rho(t,\cdot)\in S^M$ yet $E[\rho(t,\cdot)]< E[\rho_\infty^M]$ for any $t>0$, which contradicts the minimizing property of $\rho_\infty^M$ in $S^M$.
Therefore \eqref{eq:rhoDelbound} holds, implying that  
\begin{equation}
\min_{\rho\in S^M}E[\rho] = \min_{\rho\in S^{M_\cV}}E[\rho] ,\quad \forall M\ge M_\cV=\|\Delta \cV\|_{L^\infty(\BR)}.
\end{equation}  
Thus, $\rho_\infty:= \rho_\infty^{M_\cV}$ is a global minimizer,  uniformly bounded in $S^M$ for any $M\ge \|\Delta \cV\|_{L^\infty(\BR)}$.

Finally we claim that this $\rho_\infty$ is in fact a global minimizer in 
$S=\{\rho\in L^1: \rho\ge 0,\,\int \rho\rd{\bx}=m_0\}$.
 Otherwise, if there exists a $\rho\in S$ with a lower energy, $E[\rho]< E[\rho_\infty]$, then we consider the truncated  
\begin{equation}
\rho^M := \frac{m_0}{\int \min\{\rho,M\}\rd{\bx}}\min\{\rho,M\}.
\end{equation}
Then $\rho^M\in S^M$, and
\begin{equation}
E[\rho^M] \le \left(\frac{m_0}{\int \min\{\rho,M\}}\right)^2 E[\min\{\rho,M\}] \le \left(\frac{m_0}{\int \min\{\rho,M\}}\right)^2 E[\rho], 
\end{equation}
and the last quantity converges to $E[\rho]$ as $M\rightarrow\infty$. Therefore, for sufficiently large $M$, there holds $\{ \rho^M\in S^M \ : \ E[\rho^M] < E[\rho_\infty]\}$, but this contradicts the minimizing property of $\rho_\infty$ in $S^M$.
\end{proof}

\medskip\noindent
{\bf Uniqueness of steady states}. It is  straightforward to show that  global minimizers asserted in Theorem \ref{thm_exist} are unique for \emph{any} external potential $\cV(\bx)$: indeed, given any two minimizers $\rho_0$ and $\rho_1$ with the same total mass and considering the homotopy 
\begin{equation}
\rho_s(\bx) := (1-s)\rho_0(\bx) + s \rho_1(\bx) ,\quad 0\le s \le 1,
\end{equation}
one can verify the the convexity $\dfrac{\rd^2}{\rd{s}^2}E[\rho_s] > 0$, which implies uniqueness of the global energy minimizer.
However, the uniqueness of global energy minimizer does \emph{not} imply the uniqueness of steady state. In fact, a 1D example outlined in the  Appendix shows that if $\cV$ is not convex, then generally speaking steady states may not be unique, despite the uniqueness of global energy minimizer. This suggests that the conclusion of uniqueness of steady states asserted in the theorem  below is far from trivial.

\begin{theorem}\label{thm_Vuniq}
Consider the aggregation equation \eqref{eq} with radially-symmetric confinement $\cV(\bx)=\rV(r)$, \rev{satisfying} \eqref{eq:pareto} and $\Delta \cV(\bx) > 0,\,\forall \bx$. Then for each $m_0>0$, \eqref{eq} admits a unique compactly supported steady state with total mass $m_0$, and it is radially-symmetric. 
\end{theorem}
\begin{remark}
In the Appendix, consult proposition \ref{prop_cpt}, it is shown under a restrictive tail condition, $V'(r) \gtrsim r^{-\frac{d-1}{d+1}}$  for $r\ge R_0$,
that a steady solution of \eqref{eq} must be compactly supported. The gap between \eqref{eq:pareto} and this tale condition remains open.
\end{remark}

\begin{proof}
As a first step we record the following  family of radially symmetric steady states parameterized by a cut-off radius $R>0$
\[
\rho_{{}_R}(\bx) := \Delta \cV(\bx) \rchi_{|\bx|\le R}(\bx).
\]
 Indeed, the total potential field generated by $\rho_R(\bx)$
\[
\Phi_R(\bx) := \int  \cN(\bx-\by)\rho_{{}_R}(\by)\rd{\by}  + \cV(\bx) = \int  \cN(\bx-\by)\Delta \cV(\by)\rchi_{|\by|\le R}(\by)\rd{\by}  + \cV(\bx),
\]
 is radially symmetric and harmonic in $\BR$
 \[
-\Delta \Phi_R(\bx) = \Delta \cV(\bx)\rchi_{|\bx|\le R}(\bx) - \Delta \cV(\bx) = 0,\quad \forall |\bx|\le R.
\]
Therefore $\Phi_R(\bx)$ is constant in $|\bx|\le R$ and  $\bu_R= -\nabla\Phi_R$ vanishes there, 
\begin{equation}\label{nablaPhiR}
\int\nabla \cN(\bx-\by)\Delta \cV\rchi_{\BR}(\by)\rd{\by}
+\nabla \cV(\bx)=0, \qquad \forall \bx\in \BR,
\end{equation}
which means that $\rho_R=\Delta\cV\rchi_{\BR}$, satisfying \eqref{rhoinfty}, is a steady state.
Observe that this family of steady-states can be equally parametrized by their total mass:  for any  $m_0>0$, there exists a uniquely determined $R_0=R_0(m_0)>0$ such that\footnote{We make a minimal growth assumption $r^{d-1}V'(r) \stackrel{r\rightarrow \infty}{\longrightarrow}\infty$}
(${\mathbb S}^{d-1}$ denoting the  $d$-dimensional unit sphere) 
 \[
\frac{1}{|{\mathbb S}^{d-1}|}\int \Delta V \rchi_{\BRzero}\rd{\by}=\int_0^{R_0}\frac{\partial}{\partial r}\left(r^{d-1}V'(r)\right)\rd{r}= R_0^{d-1}V'(R_0)= m_0.
\]

In the second step we consider a compactly supported steady state $\rho_\infty$: we will show that it must coincide with $\rho_{{}_R}$ for properly chosen $R$.
To this end recall that  according to  \eqref{rhosuppeq}  (with $\cW=0$), a steady state of \eqref{eq} satisfies
\begin{equation}\label{rhodelV}
\rho_\infty(\bx) = \Delta \cV(\bx)\rchi_{\supp\rho_\infty}(\bx),
\end{equation}
and  by \eqref{rhoinfty} with $\cW=0$, it  is characterized by 
\begin{equation}\label{suppest1}
-\int \nabla \cN(\bx-\by)\Delta \cV(\by)\rchi_{\supp\rho_\infty}(\by)\rd{\by}  -\nabla \cV(\bx) = 0,\quad \forall \bx\in\supp\rho_\infty.
\end{equation}
Let $R_\infty$ denote its finite diameter 
$\displaystyle R_\infty = \max_{\bx\in\supp\rho_\infty}|\bx|$.
We turn to compare $\rho_\infty$ with the steady solution $\rho_{{}_{R_\infty}}= \Delta \cV \rchi_{\BRinf}$. By our first step, the latter is   a steady state, hence it also satisfies \eqref{rhoinfty} (with $\cW=0$), namely
\begin{equation}\label{suppest2}
-\int \nabla \cN(\bx-\by)\Delta \cV(\by) \rchi_{\BRinf}(\by)\rd{\by}  -\nabla \cV(\bx) = 0,\quad \forall \bx \in \BRinf.
\end{equation}
By definition, $\BRinf\supset \supp\rho_\infty$ and there exists $\bx\in  \supp\rho_\infty$ such that $|\bx|=R_\infty$. Taking the difference between \eqref{suppest1} and \eqref{suppest2} and multiply by that $\bx$ gives
\begin{equation}\label{suppest}
-\int \bx\cdot\nabla \cN(\bx-\by)\Delta \cV(\by)\rchi_{\BRinf\backslash\supp\rho_\infty}(\by)\rd{\by} = 0.
\end{equation}
Now, with $\nabla \cN(\bx) = -c_d |\bx|^{-d}\bx$ we compute that for any $|\by| < R_\infty$, consult figure \ref{fig:config} below,
\begin{equation}\label{x0y}
\bx\cdot \nabla \cN(\bx-\by) = -\frac{c_d}{|\bx-\by|^d} \bx\cdot (\bx-\by)= -\frac{c_d}{|\bx-\by|^d} (R_\infty^2 - \bx\cdot \by) < 0, \quad |\by| < R_\infty.
\end{equation}
Thus, the first integrand in \eqref{suppest} does not vanish; by assumption, the second integrand is strictly positive, and consequently the third inregrand must vanish,
\begin{equation}
\supp\rho_\infty = \{\by \ : \ |\by|\le R_\infty\}.
\end{equation}
Therefore, the steady state $\rho_\infty$ is uniquely determined as the radially symmetric $\rho_\infty=\Delta \cV(\bx) \rchi_{|\bx|\le R_\infty}(\bx)$.
\end{proof}

\medskip\noindent
{\bf Convergence rate towards  equilibrium}. A similar comparison argument has been used in \cite[\S3.1]{BLL12} in the case of quadratic potential $\cV(\bx) = |\bx|^2$. Here we extend this argument to general radially-symmetric potentials. Moreover, we pursue a considerably more intricate  comparison argument  to study the \emph{rate} of equilibration of \eqref{eq}. 
This is the content of our next result. 
\begin{theorem}\label{thm_Vequi}
Consider the aggregation equation \eqref{eq} with a $C^3$ radially-symmetric confining potential $\cV(\bx)=V(r)$, satisfying 
\begin{equation}\label{thm_Vequi_1}
0<a\le \Delta \cV(\bx) \le A < \infty,\quad \forall \bx,
\end{equation}
and subject to compactly supported initial data $\rho_0$ with uniform lower-bound \footnote{Note that $\rho_0$  is therefore discontinuous on  $\partial\supp \rho_0$ while assumed bounded away from vacuum on $\supp\rho_0$.}
\[
\rho_0(\bx)\geq \rho_{min}>0, \qquad  \forall \bx\in\supp \rho_0.
\] 
Then its energy $E(t)=E[\rho(t,\cdot)]$ decays towards the limiting energy  $E_\infty$,
\begin{equation}\label{thm_Vequi_2}
E(t) - E_\infty \le C_\gamma(1+t)^{-\gamma},\quad  \gamma<\frac{d+2}{(d-2)(d+1)},\quad  t\ge 0, \quad E_\infty=E[\rho_\infty].
\end{equation}
\rev{Furthermore, $\rho(t,\cdot)$ converges to $\rho_\infty$  with $L^1$-convergence rate
\begin{equation}
\|\rho(t,\cdot)-\rho_\infty\|_{L^1} \le C_\gamma (1+t)^{-\gamma/2}.
\end{equation}
}
\end{theorem}
The proof, provided in section \ref{sec:equi}, proceeds by comparing between the family of steady solutions,  $\rho_{{}_{R(t)}}$ with $R(t) := \max_{\bx\in\supp\rho(t,\cdot)}|\bx|$ associated with the given solution $\rho(t,\cdot)$, and the steady state $\rho_\infty$. Compared with 
the argument outlined in Theorem \ref{thm_Vuniq}, here we lack the steady state characterization  \eqref{rhodelV}: in fact, even if \eqref{rhodelV} is assumed to hold for the initial data, $\rho_0=\Delta \cV \rchi_{\supp \rho_0}$, it does not necessarily propagate in time.
We resolve this difficulty by introducing the functional 
\begin{equation}\label{F}\begin{split}
& F(t) := \frac{1}{2}\int \Big(\rho(t,\bx) - \Delta \cV(\bx)\Big)^2\rho(t,\bx)\rd{\bx},
\end{split}\end{equation}
which  measures the discrepancy of $\rho(t,\bx)$ from satisfying \eqref{rhodelV}. Then, we design a Lyapunov-type modified energy  functional,
$\widetilde{E}$ by combining $E(t) - E_\infty$, $F(t)$ and the discrepancy of radius $R(t)-R_\infty$ where 
\begin{equation}\label{R}
R(t) = \max_{\bx\in\supp\rho(t,\cdot)}|\bx|,\quad R_\infty = \max_{\bx\in\supp\rho_\infty}|\bx|.
\end{equation}
Verifying the algebraic decay rate of  $\widetilde{E}$  implies the result \eqref{thm_Vequi_2}, as well as quantifies the algebraic rate of $R(t)-R_\infty$,
\[
(R(t)-R_\infty)_+ \lesssim C_\gamma (1+t)^{-\frac{d+2}{d(d-2)(d+1)}}.
\]
The proof of Theorem \ref{thm_Vequi} tells us that the aggregation solution $\rho(t,\cdot)$ approaches the unique steady state $\rho_\infty$ in the sense of 2-Wasserstein distance with algebraic convergence rate.
Note that in the case $d=2$ this algebraic  rate $\gamma$ can be arbitrarily large, while for higher spatial dimensions, $\gamma$ is restricted by a $d$-dependent constant.
\begin{remark}
The same methodology may also apply to $\cV(\bx)$ which is not radially-symmetric, as long as the first step in our proof of Theorem \ref{thm_Vuniq} goes through. To be precise, assume the existence of a parameterized family of steady states, $\{\rho_\infty(\bx;p)\}$, such that $($i$)$ $\supp\rho_\infty(\cdot;p)$ is convex, and $($ii$)$  the following monotonicity condition holds, $\supp\rho_\infty(\cdot;p_1)\subset \supp\rho_\infty(\cdot;p_2)$ whenever $p_1<p_2$ $($and as before, there is one-to-one correspondence with the initial mass $p=p(m_0)$$)$. Then one can obtain the uniqueness of steady states for fixed $p_0$, and derive  the equilibration rate via a similar approach. It remains open to explore more general class of external potentials which give rise to the existence of such a family of steady states.
\end{remark}

\subsection{Newtonian repulsion with attraction}

We apply the ideas in the previous subsection to study the aggregation equation  \eqref{eq0},\eqref{whatisPhi} with pairwise interaction potential $\Phi$ given by sum of  Newtonian repulsion and smooth attraction potential $\cW$,
\begin{equation}\label{eq1}
\partial_t \rho + \nabla\cdot(\rho \bu) = 0,\qquad \bu(t,\bx) = -\nabla\Phi,
\quad \Phi=\cN*\rho + \cW*\rho.
\end{equation}

Observe that  being a solution of the dynamics with pairwise attraction equation \eqref{eq1}, $\rho$ can be also viewed as  a solution of the external potential  equation \eqref{eq} with a  $\rho$-dependent potential $\cV_\rho=\cW*\rho(t,\cdot)$. The distinction is that $\cV_\rho$ is time-dependent, except in the case of quadratic pairwise attraction, $\cW_2:=\hf|\bx|^2$.
Indeed, since \eqref{eq1} preserves the center of mass $\mathbf{c}_0 := \int \bx \rho_0(\bx)\rd{\bx} = \int \bx \rho(t,\bx)\rd{\bx}$, one may assume $\mathbf{c}_0=0$ without loss of generality, hence
\[
\nabla  (\cW_2*\rho)(t,\bx) = \int (\bx-\by) \rho(t,\by)\rd{\by} =m_0\bx = -\nabla \cV_2(\bx), \quad \cV_2:=\hf m_0|\bx|^2.
\]
 Thus, the forcing induced by pairwise quadratic attraction is equivalent to aggregation with quadratic confinement, $-\nabla \Phi=-\nabla\cN*\rho-\nabla\cW_2*\rho=-\nabla \cN*\rho -\nabla \cV_2$. The following theorem states the uniqueness of steady states of  pairwise attraction \eqref{eq1} for potentials, $\cW$, \emph{close to} quadratic.

\begin{theorem}\label{thm_Wuniq}
Consider the aggregation equation \eqref{eq1} with an attraction potential
\begin{equation}\label{thm_Wuniq_1}
\cW(\bx) = \frac{|\bx|^2}{2d} + \cw(\bx), \qquad |\Delta w(\bx)| \le \epsilon,
\end{equation}
where  $w(\bx)={\rw}(|\bx|)$ is a radially-symmetric perturbation  of ``order'' $\epsilon>0$, depending on $d$. Then for each $m_0>0$, \eqref{eq1} admits a unique steady state (up to translation) with total mass $m_0$,  and it is radially-symmetric.
\end{theorem}
The case $w\equiv 0$ corresponds to the Theorem  \ref{thm_Vuniq} with $\Phi=\cN*\rho+\cV_2$, Theorem \ref{thm_Wuniq} can be viewed as a perturbation of Theorem \ref{thm_Vuniq}, $\Phi=\cN*\rho+\cV$, with a perturbed potential
$\cV=\cV_2+w*\rho$, satisfying  
$\Delta \cV=d+\Delta w*\rho>1-\epsilon m_0>0$.
Alternatively, this can be viewed as aggregation driven by quadratic external forcing, $\Phi=\cN_\epsilon*\rho+\cV_2$, with perturbed Newtonian repulsion
$\cN_\epsilon:=\cN+w$.\newline
We expect that an explicit algebraic equilibration rate can be obtained by the same method as the previous subsection, and this is left as future work.

\section{Equilibration of Newtonian repulsion with confining potential}\label{sec:equi}

In this section we prove Theorem \ref{thm_Vequi}. We first prepare   a quantitative version of \eqref{x0y}.
\begin{lemma}\label{lem1}
For any $\bx$ with $|\bx|=R>0$, there holds
\begin{equation}\label{eq:lem1}
\bx\cdot \nabla \cN(\bx-\by) \le -\frac{c}{R^{d-2}} ,\quad \forall \by\ne \bx, \ |\by| \le R, \quad d\geq2.
\end{equation}
\end{lemma}

\noindent
Indeed, since $\displaystyle (\bx-\by)\cdot\bx \equiv \frac{1}{2}(|\bx-\by|^2+|\bx|^2-|\by|^2) \ge  \frac{1}{2}|\bx-\by|^2$, \eqref{eq:lem1} follows in view of 
\[
\bx\cdot \nabla \cN(\bx-\by) = -\frac{(d-2)c_d}{|\bx-\by|^d} \bx\cdot (\bx-\by) \le -\frac{(d-2)c_d}{2|\bx-\by|^{d-2}} \le -\frac{c}{R^{d-2}}, \quad c=(d-2)c_d 2^{1-d},  
\]
with the proper adjustment of $c>0$ in the 2D case. Below, we use $L^{p,q}$  denote  the usual notation of Lorentz space, e.g., \cite{BS88}.

\medskip
We will also need the following interpolation bound.  
\begin{lemma}\label{lem2} For compactly supported $g\in L^\infty_c({\Bbb R}^d)$ there holds,
\begin{equation}\label{eq:lem2}
\|g\|_{L^{d,1}} \lesssim \left\{\begin{array}{ll} C_d\|g\|^{\frac{2}{d}}_{L^2}\times \|g\|^{1-\frac{2}{d}}_{L^\infty},& \quad d>2, \\ \\ 
C_p\|g\|^{\frac{p}{2}}_{L^2}\times \|g\|^{1-\frac{p}{2}}_{L^\infty}, & \quad d=2, \forall p<2.
\end{array}\right.
\end{equation}
\end{lemma}

\noindent
Indeed, if $\lambda_g(s)=|\{\bx\, : \, |g(\bx)|>s\}|$ is the distribution function associated with $g$,  then for any $1<p<r<\infty$,
\[
\begin{split}
\|g\|_{L^{r,1}} & =r \int_0^{\|g\|_{L^\infty}} \lambda^{1/r}_g(s)\rd{s} \\
&\lesssim \left(\int_0^\infty s^p \lambda_g(s)\frac{{\rd}{s}}{s}\right)^{1/r}
\times \left(\int_0^{\|g\|_{L^\infty}} s^{-\big(\frac{p-1}{r}\big)r'}\rd{s}\right)^{1/r'}
= C_{p,r}\|g\|^{\frac{p}{r}}_{L^p}\times \|g\|^{1-\frac{p}{r}}_{L^\infty},
\end{split}
\]
and \eqref{eq:lem2} follows with $(r,p)=(d,2)$. When $d=2$ we  use it with $r=2$ and any $p<2$, so that  for compactly supported $g$'s,
\[
\begin{split}
\|g\|_{L^{2,1}}  \lesssim C_p\|g\|^{\frac{p}{2}}_{L^p}\times \|g\|^{1-\frac{p}{2}}_{L^\infty} 
\lesssim C_p\|g\|^{\frac{p}{2}}_{L^2}\times \|g\|^{1-\frac{p}{2}}_{L^\infty}, \qquad \forall p<2.
\end{split}
\]

\smallskip
\begin{proof}[Proof of Theorem \ref{thm_Vequi}]
The assumptions of Theorem \ref{thm_Vuniq} are satisfied, and hence a unique radial steady  state $\rho_\infty$ with prescribed mass $m_0$ exists, satisfying  
$\rho_\infty = \Delta \cV \rchi_{|\bx|\le R_\infty}$.

{\bf STEP 1} --- Upper and lower bounds of $\rho$. Tracing \eqref{eq}  along characteristics,
\[
\rho' := \partial_t \rho + \bu\cdot\nabla\rho = -\rho \nabla\cdot \bu = \rho(\Delta \cV-\rho), \qquad 0<a\leq \Delta \cV \leq A,
\]
 implies that after a certain time  $t_0$ (which may depend on $\displaystyle a,A\rev{,\min_{\bx\in\supp\rho_0}\rho_0(\bx)}$ but otherwise is independent\footnote{for example, take $t_0 \gtrsim \max\{|\log\tfrac{\min \rho_0}{2a}|,\frac{1}{A}\}$.} of $\max\rho_0$), there holds
\[
\frac{a}{2} \le \rho(t,\bx) \le 2A,\qquad \forall t\ge t_0,\quad \forall \bx\in \supp\rho(t,\cdot),
\]
Therefore,  by shifting the initial time if necessary, we may assume, without loss of generality, that we have the uniform bounds
\begin{equation}\label{rhomin}
0<\rho_{min} \le \rho(t,\bx) \le \rho_{max},\qquad \forall t\ge 0,\quad \forall \bx\in \supp\rho(t,\cdot).
\end{equation}

{\bf STEP 2}--- Estimate the discrepancy functional $F(t)$ in \eqref{F}.
A straightforward computation yields

\[
\begin{split}
\ddt F(t) = & \int (\rho - \Delta \cV)\partial_t\rho \cdot\rho\rd{\bx} + \frac{1}{2}\int (\rho - \Delta \cV)^2\partial_t\rho\rd{\bx} \\
= & -\int (\rho - \Delta \cV)\nabla\cdot(\rho \bu)\rho\rd{\bx} + \int (\rho - \Delta \cV)\nabla(\rho - \Delta \cV)\cdot\bu\rho\rd{\bx} \\
= & \int (-\nabla\rho\cdot \bu - \rho\nabla\cdot \bu + \nabla\rho\cdot \bu - \nabla\Delta \cV\cdot \bu)(\rho - \Delta \cV)\rho\rd{\bx} \\
= & \int \big(- \rho(\rho-\Delta \cV)  - \nabla\Delta \cV\cdot \bu\big)(\rho - \Delta \cV)\rho\rd{\bx}\\
\le & -\rho_{\min} F(t) - \int  \nabla\Delta \cV\cdot \bu(\rho - \Delta \cV)\rho\rd{\bx}.
\end{split}
\]
The second term on the right can be  bounded in terms of  the energy dissipation rate $\cD$  in \eqref{dE},
\[
\left|\int (  - \nabla\Delta \cV\cdot \bu)(\rho - \Delta \cV)\rho\rd{\bx}\right| \le \|\cV\|_{C^3}\int |\bu|\cdot|\rho - \Delta \cV|\rho\rd{\bx} \le \|\cV\|_{C^3}\Big( \frac{\rho_{min}}{2\|\cV\|_{C^3}}F + \frac{2\|\cV\|_{C^3}}{\rho_{min}}\cD \Big),
\]
and we end up with 
$\displaystyle \ddt F(t) \le -\frac{\rho_{min}}{2} F + C\cD$.
This implies that $F$ is bounded: in fact, since  $\cD=-\ddt E$ it follows that $F+C(E-E_\infty) \leq F_0+C(E_0-E_\infty)$. Hence we seek the large time behavior for  quantities $F, (E-E_\infty)$ (and likewise $R-R_\infty$ in the next step) which depending on their vanishing order $\ll 1$.
Observe  with small enough $\epsilon_1>0$ there follows 
\begin{equation}\label{dF}
\ddt \Big((E(t)-E_\infty) + \epsilon_1 F(t)\big)
\leq -\cD +\epsilon_1\left(-\frac{\rho_{min}}{2}F +C\cD\right) \leq -c(\cD+F).
\end{equation}
To close this inequality, we will need to take into account the further discrepancy between  $\supp \rho(t,\cdot)$ and $\supp \rho_\infty$.

{\bf STEP 3} --- Estimate of $R'(t)$. Recall that $R(t)$ is the radius of $\supp \rho(t,\cdot)$, \eqref{R} and assume for a moment that $R(t)\ge R_\infty$, see figure \ref{fig:config} for a typical configuration\footnote{Note that $\supp \rho_0$ and hence $\supp \rho$ need \emph{not} be  simply connected.}.

\begin{figure}[ht]
		\centering
		\includegraphics[width=.65\linewidth]{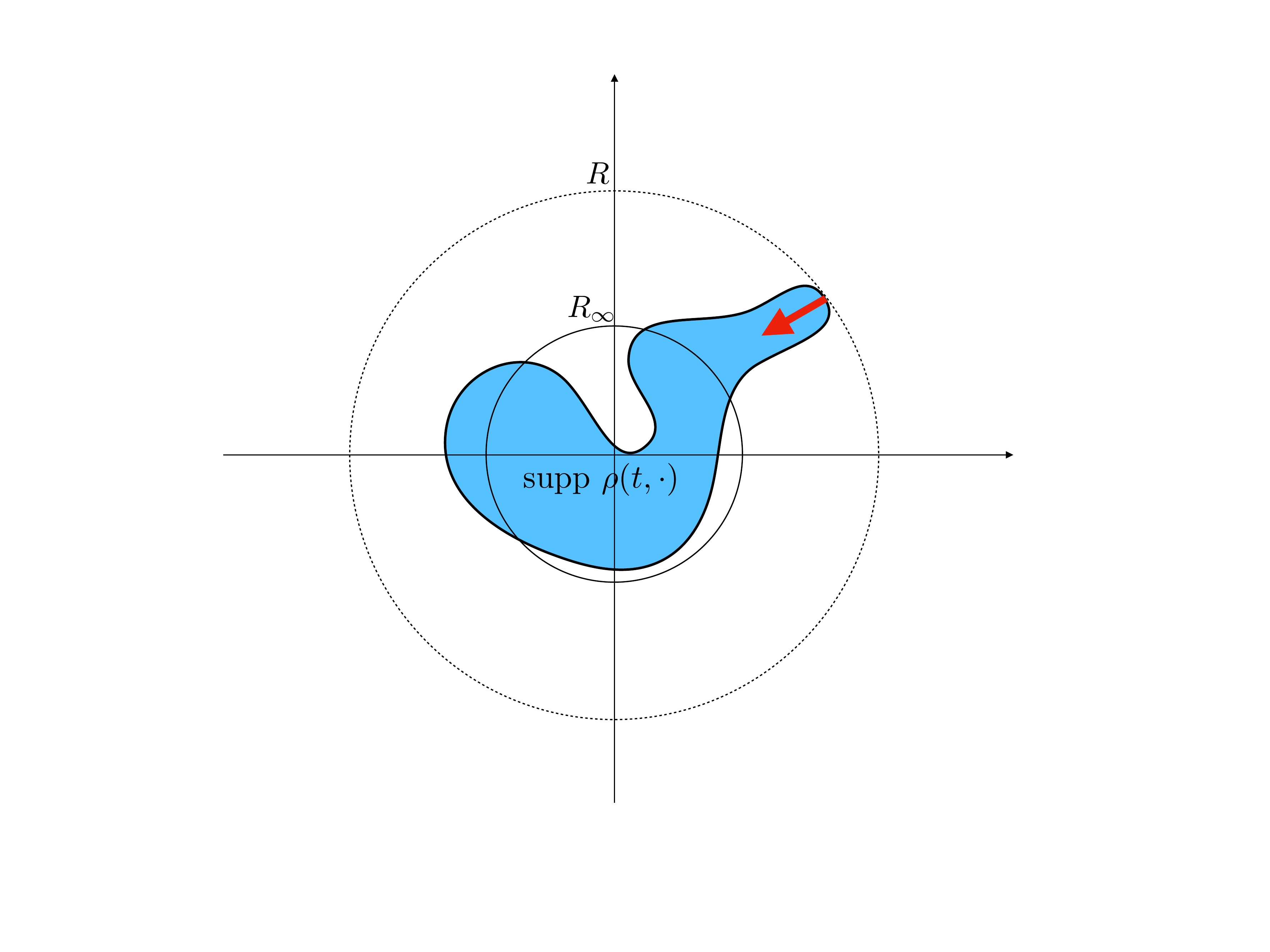}  
		\vspace*{-1.3cm}
		\caption{The support of $\rho(t,\cdot)$ inscribed in $\BR$ vs. the limiting ball $\BRinf$.}
		\label{fig:config}
\end{figure}

Fix $\bx$ on the  edge of $\supp \rho(t),  |\bx|=R$. Then
by \eqref{nablaPhiR} the velocity $\bu$ in \eqref{whatisPhi} amounts to
\[
\begin{split}
 \bu(t,\bx) = & -  \int_{|\by|\le R}  \nabla\cN(\bx-\by)\rho(t,\by)\rd{\by} -\nabla \cV(\bx) \\ 
= & - \int_{|\by|\le R} \nabla \cN(\bx-\by)(\rho(t,\by)-\Delta \cV(\by))\rd{\by} 
    - \left( \int_{|\by|\le R} \nabla \cN(\bx-\by)\Delta \cV(\by)\rd{\by} + \nabla \cV(\bx) \right)\\
    =& -\int_{|\by|\le R} \nabla \cN(\bx-\by)\big(\rho(t,\by)-\Delta \cV(\by)\big)\rd{\by}.
\end{split}
\]
We estimate   the last term  by examining  separately\footnote{Here and below we let $z_-,z_+$ denote the negative and receptively positive parts of a real $z$.}, $\bu_\pm := -\nabla \cN*\big((\rho-\Delta\cV)_\pm \rchi_{\BR}\big)$. We begin by estimating the  discrepancy from below, $(\rho-\Delta\cV)_-\rchi_{\BR}$. By Lemma \ref{lem1}, 
\begin{equation}\label{eq:uminus}
\begin{split}
\bx\cdot \bu_-(t,\bx) & = -\int_{|\by|\le R} \bx\cdot \nabla \cN(\bx-\by)\big(\rho(t,\by)-\Delta \cV(\by)\big)_{-}\rd{\by}  \\
 & \le  \frac{c}{R^{d-2}}\int_{|\by|\le R}\big(\rho(t,\by)-\Delta \cV(\by)\big)\rd{\by} \\
& =  \frac{c}{R^{d-2}}\left(\int_{|\by|\le R_\infty}\Delta \cV\rd{\by} - \int_{|\by|\le R}\Delta \cV\rd{\by}\right) \\
& =  -\frac{c}{R^{d-2}}\int_{R_\infty \le |\by| \le R}\Delta \cV \rd{\by}  \\
& \le  - \frac{c}{R^{d-2}}\frac{a}{d}(R^d-R^d_\infty )\\
& \lesssim   -R(R-R_\infty ),
\end{split}
\end{equation}
where the second equality uses the fact that $\displaystyle \int_{|\by|\le R} \rho(\by)\rd{\by} =m_0= \int_{|\by|\le R_\infty}\Delta \cV\rd{\by}$ and  the second inequality uses the lower bound $\Delta \cV\geq a$.\newline

Next, we estimate the  discrepancy from above, $g=(\rho-\Delta\cV)_+$. 
Since $\nabla \cN \in L^{d',\infty}$ then $\|\nabla\cN*g\|_{L^\infty} \lesssim \|g\|_{L^{d,1}}$. 
Recall that $g$ is uniformly bounded, supported in $\BR$ and satisfies the $L^2$ bound $\|g\|^2_{L^2}\leq 1/\rho_{min}F(t)$, so Lemma \ref{lem2} implies the existence of finite $C_d, C_p$ such that
\begin{equation}\label{eq:uplus}
\begin{split}
\frac{\bx}{R}\cdot \bu_+(t,\bx)  = -\int \frac{\bx}{R}\cdot &\nabla \cN(\bx-\by)\big(\rho(t,\by)-\Delta \cV(\by)\big)_+\rd{\by} \\
 & \leq \|\rho(t,\cdot)-\Delta \cV\|_{L^{d,1}} \leq 
 \left\{
 \begin{array}{ll} C_d\left(F(t)\right)^{1/d}, & d>2\\ \\
   C_p\left(F(t)\right)^{{p}/{4}}, & \forall p<d=2. \end{array}
   \right.
\end{split}
\end{equation}
Using the bounds \eqref{eq:uminus},\eqref{eq:uplus} we find
\[
\ddt(R(t)-R_\infty)_{+} =\sup_{|\bx|=R,\,\bx\in\supp\rho} \bu(t,\bx) \cdot \frac{\bx}{R}
\le  -c(R(t)-R_\infty)_{+} + C\big(F(t)\big)^{1/s},
\]
with $s:= \left\{\begin{array}{ll} d, & d>2\\ 
\nicefrac{4}{p}, & \forall p<d=2.\end{array} \right\}>2$.
We note on passing that since $F$ is bounded (due to \eqref{dF}), hence $(R(t)-R_\infty)_+$ remains bounded. We proceed to show its time decay.\newline
Fix  an \emph{arbitrary} $m>d$. By Young's inequality we have
\begin{equation}\label{eq:Young}
\begin{split}
\rev{\ddt(R(t)-R_\infty)^m_{+} \leq  -cm(R(t)-R_\infty)^m_{+}  +  \frac{\left(\delta m\right)^{s'}}{s'}(R(t)-R_\infty)^{(m-1)s'}_{+} + \frac{1}{s}\left(\frac{C}{\delta}\right)^sF(t).}
\end{split}
\end{equation}
Note that since  $m>d$ then $(m-1)s' >m$: indeed, when $d>2$ then $s=d$ and $d'>m'=m/(m-1)$, and  when $d=2$ then  we can always choose $p$  so that   $4/m < p< 2$ and with $s=4/p$ we then have $(m-1)(4/p)'>m$. Therefore,  choosing small enough $\delta>0$, makes   the first term on the right of \eqref{eq:Young} dominates the second for bounded $R(t)$'s, and we conclude  \rev{the existence of large enough $C_\delta>0$ depending on $(m,s,R_\infty)$, such that}
\begin{equation}\label{dR}
\begin{split}
\rev{\ddt(R(t)-R_\infty)^m_{+} \leq  -c_m(R(t)-R_\infty)^m_{+} + C_\delta F(t), \quad c_m:=\hf cm \qquad \forall m>d, \ d\geq 2.}
\end{split}
\end{equation}

{\bf STEP 4} ---  We form the Lyapunov functional, $\widetilde{E}(t)$, as  a suitable linear combination of  
 \[
 \widetilde{E}(t):=(E(t)-E_\infty) + \epsilon_1 F(t) + \epsilon_2(R(t)-R_\infty)^m_{+},
 \]
  with fixed $\epsilon_1 \gg \epsilon_2>0$ which are yet to be  chosen.
 Choosing  the corresponding combination of \eqref{dE}, \eqref{dF} and \eqref{dR}
  with small enough $\epsilon_2$ then yields, 
\begin{equation}\label{decay1}
\begin{split}
\frac{\rd}{\rd{t}}\widetilde{E}  & \le -c(\cD+F) -c_m\epsilon_2 (R-R_\infty)^m_{+} + C\epsilon_2 F \leq -\hf cF-c\epsilon_2(R-R_\infty)^m_+.
\end{split}
\end{equation}
with (re-labeled) constants $0 < c \ll 1 \ll C$ which are independent of $\epsilon_2$.\newline

{\bf STEP 5} --- Close the estimate. We aim to show that 
\begin{equation}\label{claim1}
\begin{split}
E[\rho(t)]-E_\infty \leq C_q\big((R(t)-R_\infty)^{2/q}_{+} + F(t)\big), \qquad  \left\{\begin{array}{ll} q=\frac{2d}{d+2}, & d> 2\\ \\
                           \text{any} \ q>1, & d=2. \end{array}\right.
\end{split}
\end{equation}
Combined with \eqref{decay1}, we obtain, noticing that $\alpha:=\frac{m}{2/q}>1$ and adjusting $\epsilon_2 \ll 1$ if necessary,
\[
\ddt \widetilde{E}  \le -c\widetilde{E}^{\alpha}, \qquad \alpha=\frac{mq}{2}
> \left\{\begin{array}{ll} d\frac{d}{d+2}, & d>2\\ \\
1, & d=2,\end{array}\right.
\]
which recovers  \eqref{thm_Vequi_2}, $E(t)-E_\infty \le \widetilde{E} \lesssim (1+t)^{-\gamma}$ with $\gamma=1/(\alpha-1)$. 

It remains to prove \eqref{claim1}. Let $\rho_1$ denote the discrepancy of $\rho$ from the steady state $\rho_\infty=\Delta \cV\rchi_{\BRinf}$,
\begin{equation}\label{rho1}
\rho_1 :=  \rho- \Delta \cV \rchi_{\BRinf}, \qquad \int \rho_1 \rd{\bx} = 0.
\end{equation}
Observe that $\rho_1$ is  uniformly bounded since $\Delta\cV$ and $\rho$ are, and that is supported in $\BR$; more precisely
$\rho_1 = \rho\rchi_{\BR\backslash \BRinf} -(\Delta\cV-\rho)\rchi_{\BRinf}$ hence
\rev{\begin{equation}
\begin{split}
\|\rho_1(t,\cdot)\chi_{\supp\rho}\|_{L^1} & =\int_{\BR\backslash \BRinf} \rho\rd{\bx} + \int_{\BRinf\cap\supp\rho}|\Delta \cV-\rho|\rd{\bx}
 \\
&\lesssim C\rho_{\max} (R-R_\infty) + \left(\frac{R^d_\infty}{\rho_{min}}\right)^{1/2}\left(\int|\Delta \cV-\rho|^2\rho(t,\bx)\rd{\bx}\right)^{1/2}  \\
& \lesssim (R(t)-R_\infty)_+ +F^{1/2}(t).
\end{split}
\end{equation}
This implies that 
\begin{equation}\label{rho1L1}
\|\rho_1(t,\cdot)\|_{L^1}\lesssim (R(t)-R_\infty)_+ +F^{1/2}(t),
\end{equation}
by the mean-zero property of $\rho_1$, since $\rho_1$ is nonpositive on $(\supp\rho)^c$.
}

Expressed in terms of $\rho_1$, the discrepancy of the energy is given by
\begin{equation}\label{Erhoinf}\begin{split}
E[\rho]-E_\infty = \int \Phi_\infty(\bx)\rho_1(\bx)\rd{\bx} + \frac{1}{2}\iint \cN(\bx-\by) \rho_1(\bx)\rho_1(\by)\rd{\bx}\rd{\by}.
\end{split}\end{equation}
Let us first bound the first linear term on the right of \eqref{Erhoinf}. 
Here
$\displaystyle \Phi_\infty(\bx) := \int \cN(\bx-\by)\Delta \cV(\by) \rchi_{\BRinf}(\by)\rd{\by} + \cV(\bx)$ is the total potential generated by the steady state and as before, being radial and harmonic it remains constant in $\BRinf$. Let $\Phi_\infty(R_\infty\frac{\bx}{|\bx|})$ be the radial extension of this constant throughout $\BR$: since $\rho_1$ has zero mean on $\BR$ then $\displaystyle \int_{\BR}\Phi_\infty\Big(R_\infty\frac{\bx}{|\bx|}\Big)\rho_1(\bx)\rd{\bx}=0$, and  
since $\Phi_\infty(\bx)$ is Lipschitz outside $\BRinf$   (because we assume that $\Delta \cV$ is), then  \eqref{rho1L1}  implies
\begin{equation}\label{Phiinfty}\begin{split}
\left|\int \Phi_\infty(\bx)\rho_1(\bx)\rd{\bx}\right|  &=  \left|\int_{\BR} \Big(\Phi_\infty(\bx)-\Phi_\infty\Big(R_\infty\frac{\bx}{|\bx|}\Big)\Big)\rho_1(\bx)\rd{\bx}\right| \\
 &=  \left|\int_{\BR\backslash \BRinf} \Big(\Phi_\infty(\bx)-\Phi_\infty\Big(R_\infty\frac{\bx}{|\bx|}\Big)\Big)\rho_1(\bx)\rd{\bx}\right| \\
 & \lesssim (R-R_\infty)_{+}\|\rho_1\|_{L^1} \\
  & \lesssim (R-R_\infty)^2_+  + F(t).
\end{split}
\end{equation}
To estimate the quadratic term in \eqref{Erhoinf}, we separate between the cases $d>2$ and $d=2$. For  the former, set $q=\dfrac{2d}{d+2}\in (1,2)$ and use  Hardy-Littlewood-Sobolev with $\cN\in L^{\frac{d}{d-2},\infty}$ to conclude
\begin{equation}\label{Qrhoinf}
\begin{split}
\Big|\int \cN(\bx-\by)&\rho_1(\bx)\rho_1(\by)\rd{\bx}\rd{\by}\Big| \\
& \lesssim \|\rho_1\|^2_{L^{q}}   \lesssim \left( \int_{\BR\backslash \BRinf}\rho^q\rd{\bx}\right)^{\frac{2}{q}} + \left(\int_{\BRinf}|\Delta \cV-\rho|^q\rd{\bx}\right)^{\frac{2}{q}} \\
 & \leq C\rho^2_{max}(R-R_\infty)^{\frac{2}{q}}_+ +\left(\frac{R^d_\infty}{\rho_{\min}}\right)^{\dfrac{2/q}{(2/q)'}}\int_{\BRinf}|\Delta \cV-\rho|^2\rho\rd{\bx} \\
  & \lesssim (R-R_\infty)^{\frac{2}{q}}_{+} + F, \qquad q=\frac{2d}{d+2}\in (1,2).
\end{split}
\end{equation}
For the remaining case $d=2$ we recall that  $\rho_1$ has zero mean,  hence
 the 2D embedding $\|\rho_1\|_{\dot{H}^{-1}} \leq C_q \|\rho_1\|_{L_{\it loc}^q}$ recovers \eqref{Qrhoinf} for \emph{any} $q>1$ 
\[
\Big|\int \cN(\bx-\by)\rho_1(\bx)\rho_1(\by)\rd{\bx}\rd{\by}\Big|=\|\rho_1\|_{\dot{H}^{-1}}^2 \lesssim C_q\|\rho_1\|_{L_{\it loc}^q}^2 \lesssim (R-R_\infty)^{\frac{2}{q}}_{+} + F, \qquad \forall q>1.
\]
Now \eqref{claim1} follows from \eqref{Erhoinf},\eqref{Phiinfty} and \eqref{Qrhoinf}.

\rev{Finally, to show the convergence of $\rho(t,\cdot)$ to $\rho_\infty$, we notice that the energy estimate $ \widetilde{E} \lesssim (1+t)^{-\gamma}$ implies
\begin{equation}
(R(t)-R_\infty)_+ \lesssim (1+t)^{-\gamma q/2},\quad F(t) \lesssim (1+t)^{-\gamma}
\end{equation}
where $q$ is as defined in \eqref{claim1}. Therefore \eqref{rho1L1} implies
\begin{equation}
\|\rho_1(t,\cdot)\|_{L^1} = \|\rho(t,\cdot)-\rho_\infty\|_{L^1} \lesssim  (1+t)^{-\gamma q/2} + (1+t)^{-\gamma/2} \lesssim (1+t)^{-\gamma/2}.
\end{equation}
}
\end{proof}

\section{Uniqueness of steady state for Newtonian repulsion with near-quadratic attraction}
First notice that \eqref{thm_Wuniq_1} implies that for any $r>0$,
\begin{equation}
\left|\rw'(r)\int_{|\bx|=r}\rd{S} \right| = \left|\int_{|\bx|=r} \frac{\bx}{|\bx|}\cdot \nabla w(\bx) \rd{S}\right| =  \left|\int_{|\bx|\le r} \Delta w(\bx)\rd{\bx}\right| \le \epsilon |{\mathbb B}_r|,
\end{equation}
Therefore
\begin{equation}\label{thm_Wuniq_11}
|\rw'(r)| \le \epsilon\cdot \frac{r}{d}.
\end{equation}

\begin{proof}[Proof of Theorem \ref{thm_Wuniq}]

Let $\rho_\infty$ be the global energy minimizer of $E[\rho]$ among all radially-symmetric density distributions with total mass $m_0$. Since the gradient flow \eqref{eq1} preserves the radial symmetry, $\rho_\infty$ is clearly a steady state of \eqref{eq1}.

Assume $\rho(\bx)$ is a steady state of \eqref{eq1} with total mass $m_0$ (and assume its center of mass $\int \bx\rho(\bx)\rd{\bx}=0$ without loss of generality), and we aim to show $\rho=\rho_\infty$.

Denote $R=\max_{\bx\in\supp\rho}|\bx|$ and let
\begin{equation}
\tV(\bx) = \int \cW(\bx-\by) \rho(\by)\rd{\by},\quad \tV_\infty(\bx) = \int \cW(\bx-\by) \rho_\infty(\by)\rd{\by},
\end{equation}
be the attractive potential fields generated by $\rho$ and $\rho_\infty$. Here $\tV_\infty$ is radially-symmetric because $\rho_\infty$ is. Then $\rho(\bx)$ is a steady state of \eqref{eq} with $\cV$ replaced by $\tV$, which implies
\begin{equation}\label{tV}
\rho=\Delta \tV \rchi_{\supp\rho},\quad -\int \nabla \cN(\bx-\by)\Delta \tV(\by)\rchi_{\supp\rho}(\by)\rd{\by}  -\nabla \tV(\bx) = 0,\quad \forall \bx\in\supp\rho.
\end{equation}
Similarly
\begin{equation}
\begin{split}
\rho_\infty& =\Delta \tV_\infty \rchi_{\supp\rho_\infty}: \\
 & \qquad -\int \nabla \cN(\bx-\by)\Delta \tV_\infty(\by)\rchi_{\supp\rho_\infty}(\by)\rd{\by}  -\nabla \tV_\infty(\bx) = 0,\quad \forall \bx\in\supp\rho_\infty.\end{split}
\end{equation}

 The assumptions on $\cW$ imply that
\begin{equation}
1-\epsilon \le \Delta \cW(\bx) \le 1+\epsilon,\quad \forall \bx,
\end{equation}
and therefore
\begin{equation}\label{DeltaV}
m_0(1-\epsilon) \le \Delta \tV(\bx) \le m_0(1+\epsilon), \quad m_0(1-\epsilon) \le \Delta \tV_\infty(\bx) \le m_0(1+\epsilon).
\end{equation}

Next we compute 
\begin{equation}\label{Vminus0}\begin{split}
\tV(\bx) - \tV_\infty(\bx) = & \int w(\bx-\by) \rho(\by)\rd{\by}\\
 &  - \int w(\bx-\by) \rho_\infty(\by)\rd{\by} \\
= & \int w(\bx-\by) \Delta \tV (\by) \rchi_{\supp\rho}(\by)\rd{\by} - \int w(\bx-\by) \Delta \tV_\infty(\by)\rchi_{\supp\rho_\infty}(\by)\rd{\by} \\
= & \int w(\bx-\by) (\Delta \tV (\by)-\Delta \tV_\infty(\by)) \rchi_{\supp\rho\cap \supp\rho_\infty}(\by)\rd{\by} \\
& + \int w(\bx-\by) \Delta \tV (\by) \rchi_{\supp\rho\backslash\supp\rho_\infty}(\by)\rd{\by} \\
& - \int w(\bx-\by) \Delta \tV_\infty(\by)\rchi_{\supp\rho_\infty\backslash\supp\rho}(\by)\rd{\by} \\
=: & I_1 + I_2 + I_3 .
\end{split}\end{equation}

{\bf STEP 1} --- estimate $\|\Delta \cV - \Delta \cV_\infty\|_{L^\infty}$.

We take the Laplacian of \eqref{Vminus0}:
\begin{equation}\label{Vminus}
\Delta \tV(\bx) - \Delta \tV_\infty(\bx) = \Delta I_1 + \Delta I_2 + \Delta I_3,
\end{equation}
 and estimate the three terms on the RHS.
\begin{equation}\begin{split}
|\Delta I_1| = & \left|\int \Delta w(\bx-\by) \big(\Delta \tV (\by)-\Delta \tV_\infty(\by)\big) \rchi_{\supp\rho\cap \supp\rho_\infty}(\by)\rd{\by} \right| \\
\le & \epsilon\cdot |\supp\rho_\infty|\cdot \|\Delta \cV - \Delta \cV_\infty\|_{L^\infty} , \\
\end{split}\end{equation}
by \eqref{thm_Wuniq_1}.

\begin{equation}
\begin{split}
|\Delta I_2| & = \left|\int \Delta w(\bx-\by) \Delta \tV (\by) \rchi_{\supp\rho\backslash\supp\rho_\infty}(\by)\rd{\by}\right| \\
& \le  \epsilon\cdot m_0(1+\epsilon) \cdot \Big|\supp\rho\backslash\supp\rho_\infty\Big|, 
\end{split}\end{equation}
by \eqref{thm_Wuniq_1} and \eqref{DeltaV}.

To estimate $I_3$, we first use the fact that $\rho$ and $\rho_\infty$ have the same total mass, and obtain
\begin{equation}\begin{split}
0 = & \int \rho(\bx)\rd{\bx} - \int \rho_\infty(\bx)\rd{\bx} \\
= &  \int \Delta \tV(\bx)\rchi_{\supp\rho}(\bx)\rd{\bx} - \int \Delta \tV_\infty(\bx)\rchi_{\supp\rho_\infty}(\bx)\rd{\bx} \\
= & \int ( \Delta \tV(\bx) - \Delta \tV_\infty(\bx)) \rchi_{\supp\rho\cap \supp\rho_\infty}(\bx) \rd{\bx} \\
 & \ \ +  \int \Delta \tV(\bx)\rchi_{\supp\rho\backslash\supp\rho_\infty}(\bx)\rd{\bx} - \int \Delta \tV_\infty(\bx)\rchi_{\supp\rho_\infty\backslash\supp\rho}(\bx)\rd{\bx} .\\
\end{split}\end{equation}
Therefore
\begin{equation}\begin{split}
& \left|\int \Delta \tV_\infty(\bx)\rchi_{\supp\rho_\infty\backslash\supp\rho}(\bx)\rd{\bx}\right| \\
 & \quad =  \left| \int ( \Delta \tV(\bx) - \Delta \tV_\infty(\bx)) \rchi_{\supp\rho\cap \supp\rho_\infty}(\bx) \rd{\bx} +  \int \Delta \tV(\bx)\rchi_{\supp\rho\backslash\supp\rho_\infty}(\bx)\rd{\bx} \right| \\
& \quad \le  |\supp\rho_\infty|\cdot \|\Delta \tV - \Delta \tV_\infty\|_{L^\infty} + m_0(1+\epsilon)\cdot \Big|\supp\rho\backslash\supp\rho_\infty\Big|.
\end{split}\end{equation}
This implies
\begin{equation}\begin{split}
|\Delta I_3| = & \left|\int \Delta w(\bx-\by) \Delta \tV_\infty(\by)\rchi_{\supp\rho_\infty\backslash\supp\rho}(\by)\rd{\by}\right| \\
\le &  \epsilon\cdot |\supp\rho_\infty|\cdot \|\Delta \tV - \Delta \tV_\infty\|_{L^\infty} + \epsilon\cdot m_0(1+\epsilon)\cdot \Big|\supp\rho\backslash\supp\rho_\infty\Big|.
\end{split}\end{equation}

Finally, use these in \eqref{Vminus} we conclude that 
\begin{equation}\label{VV10}
\|\Delta \tV - \Delta \tV_\infty\|_{L^\infty} \le 2\epsilon \cdot |\supp\rho_\infty|\cdot  \|\Delta \tV - \Delta \tV_\infty\|_{L^\infty} + 2\epsilon\cdot m_0(1+\epsilon)\cdot \Big|\supp\rho\backslash\supp\rho_\infty\Big|.
\end{equation}
If $\epsilon$ is small enough so that $|\supp\rho_\infty|\cdot 2\epsilon < 1$, then
\begin{equation}\label{VV1}
\|\Delta \tV - \Delta \tV_\infty\|_{L^\infty} \le \frac{2\epsilon\cdot m_0(1+\epsilon)}{1-|\supp\rho_\infty|\cdot 2\epsilon}  \cdot \Big|\supp\rho\backslash\supp\rho_\infty\Big|.
\end{equation}

As a byproduct, this shows that unless $\Delta \tV - \Delta \tV_\infty=0$ which implies the conclusion, we always have $\supp\rho\not\subset\supp\rho_\infty = \{\bx:|\bx|\le R_\infty\}$ and therefore $R> R_\infty$. Now we will show that the option $R> R_\infty$ is impossible.

{\bf STEP 2} --- use comparison principle.  Assume on the contrary that $R>R_\infty$. Taking $\nabla$ on \eqref{Vminus0} and conducting similar estimates gives
\begin{equation}\label{VV2}
\begin{split}
| \nabla \tV(\bx)-\nabla \tV_\infty(\bx)|  \le  \epsilon \cdot \frac{2R}{d} \cdot 2\Big(&|\supp\rho_\infty|\cdot \|\Delta \tV - \Delta \tV_\infty\|_{L^\infty} \\
&  +  m_0(1+\epsilon) \cdot \Big|\supp\rho\backslash\supp\rho_\infty\Big|\Big),\qquad \forall |\bx| \le R,
\end{split}
\end{equation}
using $|\nabla w(\bx-\by)| \le \epsilon\frac{|\bx-\by|}{d} \le \epsilon \cdot \frac{2R}{d}$ by \eqref{thm_Wuniq_11}.

The fact that $\Delta \tV_\infty \rchi_{|\bx|\le R}$ is a steady state of \eqref{eq} with $\tV_\infty$ implies
\begin{equation}
-\int \nabla \cN(\bx-\by)\Delta \tV_\infty(\by)\rchi_{|\by|\le R}(\by)\rd{\by}  -\nabla \tV_\infty(\bx) = 0,\quad \forall |\bx|\le R.
\end{equation}
Taking difference with \eqref{tV} and evaluating at $\bx\in\supp\rho$ with $|\bx|=R$ (such an $\bx$ exists due to the definition of $R$) gives
\begin{equation}\label{VV3}
\begin{split}
 -\int_{|\by|\le R} &\nabla \cN(\bx-\by)\big(\Delta \tV_\infty(\by) - \rho(\by)\big)_+\rd{\by} \\
& -\int_{|\by|\le R} \nabla \cN(\bx-\by)\big(\Delta \tV_\infty(\by) - \rho(\by)\big)_-\rd{\by} -\big(\nabla \tV_\infty(\bx)-\nabla \tV(\bx)\big) = 0.
\end{split}
\end{equation}
Since $\supp\rho\subset \BR$ and $\rho=\Delta \tV \rchi_{\supp\rho}$, 
\begin{equation}
|\big(\Delta \tV_\infty(\by) - \rho(\by)\big)_-| \le \|\Delta \tV_\infty-\Delta \tV\|_{L^\infty},\quad \forall |\by|\le R.
\end{equation}
Also notice that since $R\ge R_\infty$, we have $\int_{|\by|\le R}\rho(\by)\rd{\by} = m_0 = \int_{|\by|\le R_\infty}\Delta \tV_\infty(\by)\rd{\by}$, which implies
\[
\begin{split}
\int_{|\by|\le R} & \Delta \tV_\infty(\by) \rd{\by} - \int_{|\by|\le R}\rho(\by)\rd{\by} \\
& = \int_{|\by|\le R} \Delta \tV_\infty(\by) \rd{\by} - \int_{|\by|\le R_\infty}\Delta \tV_\infty(\by)\rd{\by} \ge m_0(1-\epsilon) |\{R_\infty \le |\by|\le R\}|.
\end{split}
\]
Therefore
\begin{equation}\label{VV40}
\int_{|\by|\le R} \big(\Delta \tV_\infty(\by) - \rho(\by)\big)_+\rd{\by} \ge m_0(1-\epsilon) |\{R_\infty \le |\by|\le R\}|.
\end{equation}

Take inner product of \eqref{VV3} with $\bx$. Lemma \ref{lem1} with \eqref{VV40} shows that 
\begin{equation}\label{VV4}\begin{split}
-\bx\cdot \int_{|\by|\le R} \nabla \cN(\bx-\by)\big(\Delta \tV_\infty(\by) - \rho(\by)\big)_+\rd{\by} \ge & \frac{c_d}{R^{d-2}}\cdot m_0(1-\epsilon) |\{R_\infty \le |\by|\le R\}|.
\end{split}\end{equation}
Then we estimate the other two terms in \eqref{VV3}, after taking inner product with $\bx$: 
\begin{equation}\label{VV5}\begin{split}
& \left| \bx\cdot \int_{|\by|\le R} \nabla \cN(\bx-\by)\big(\Delta \tV_\infty(\by) - \rho(\by)\big)_-\rd{\by} -\bx\cdot \big(\nabla \tV_\infty(\bx)-\nabla \tV(\bx)\big)\right| \\
\le & \|\Delta \tV - \Delta \tV_\infty\|_{L^\infty}\cdot  \int  (-\bx)\cdot\nabla \cN(\bx-\by) \rchi_{|\by|\le R}(\by)\rd{\by} + R| \nabla \tV(\bx)-\nabla \tV_\infty(\bx)| \\
\le & \|\Delta \tV - \Delta \tV_\infty\|_{L^\infty}\cdot \frac{R^2}{d} +  R\cdot \epsilon \cdot \frac{2R}{d} \cdot 2\Big(|\supp\rho_\infty|\cdot \|\Delta \tV - \Delta \tV_\infty\|_{L^\infty} \\
& +  m_0(1+\epsilon) \cdot \Big|\supp\rho\backslash\supp\rho_\infty\Big|\Big) \\
\le & 2\epsilon\cdot m_0(1+\epsilon)\frac{R^2}{d}\cdot \left(\frac{1 + 4\epsilon \cdot |\supp\rho_\infty|}{1-|\supp\rho_\infty|\cdot 2\epsilon} + \rev{2} \right)\cdot \Big|\supp\rho\backslash\supp\rho_\infty\Big| \\
\le & 2\epsilon\cdot m_0(1+\epsilon)\frac{R^2}{d}\cdot \left(\frac{1 + 4\epsilon \cdot |\supp\rho_\infty|}{1-|\supp\rho_\infty|\cdot 2\epsilon} +  \rev{2} \right)\cdot \min\{|\{R_\infty \le |\by|\le R\}|,|\supp\rho|\},
\end{split}\end{equation}
where the first inequality uses the fact that $(-\bx)\cdot\nabla \cN(\bx-\by)\ge 0$ by Lemma \ref{lem1}, the second inequality uses \eqref{VV2} and the fact that $\rchi_{|\by|\le R}$ is a steady state of \eqref{eq} with $\cV(\bx) = |\bx|^2/(2d)$, and the third inequality uses \eqref{VV1}.

If $R\le 2R_\infty$, then \eqref{VV4} and \eqref{VV5} contradict \eqref{VV3}. In fact, if $R>R_\infty$, and $\epsilon$ is small enough such that
\begin{equation}\label{eps1}
 2\epsilon\cdot \frac{1+\epsilon}{1-\epsilon}\cdot \frac{1}{d}\left(\frac{1 + 4\epsilon \cdot |\supp\rho_\infty|}{1-|\supp\rho_\infty|\cdot 2\epsilon} +  \rev{2} \right) < \frac{c_d}{(2R_\infty)^{d}},
\end{equation}
then the RHS of \eqref{VV4} is greater than that of \eqref{VV5}, which gives the contradiction.

If $R>2R_\infty$, then by the estimates
\begin{equation}
|\supp\rho| \le \frac{1}{1-\epsilon},\quad  |\{R_\infty \le |\by|\le R\}| \ge \frac{2^d-1}{2^d}|{\mathbb B}_1|\cdot R^d,\quad \forall R>2R_\infty.
\end{equation}
\eqref{VV4} and \eqref{VV5} contradict \eqref{VV3}, if $\epsilon$ is small enough such that
\begin{equation}\label{eps2}
 2\epsilon\cdot \frac{1+\epsilon}{(1-\epsilon)^2}\cdot \frac{1}{d}\left(\frac{1 + 4\epsilon \cdot |\supp\rho_\infty|}{1-|\supp\rho_\infty|\cdot 2\epsilon} +  \rev{2} \right) < c_d\frac{2^d-1}{2^d}.
 \end{equation}

Notice the estimate
\begin{equation}
|\supp\rho_\infty| \le \frac{1}{1-\epsilon},\quad R_\infty \le \frac{c_d}{(1-\epsilon)^{1/d}},
\end{equation}
which implies the smallness conditions \eqref{eps1} and \eqref{eps2} on $\epsilon$ only depend on $d$.

\end{proof}

\begin{remark}
Compared to the proof of Theorem \ref{thm_Vuniq}, the main new ingredient in the above proof is a contraction argument, which can be seen in the derivation from \eqref{VV10} to \eqref{VV1}.
\end{remark}

\section{Appendix}
\subsection{1D steady state are not unique}
In the Appendix we give a description of the steady states \eqref{eq} when $d=1$. In this case, one can write \eqref{eq} as 
\begin{equation}\label{eq1D}
\partial_t \rho +\partial_x(\rho u) = 0,\quad u(t,x) = -\int  \cN'(x-y)\rho(t,y)\rd{y}  - \cV'(x).
\end{equation}
Define $m(t,x)$ as the primitive of $\rho(t,x)$:
\[
m(t,x) := \int_{-\infty}^x \rho(t,y)\rd{y} - \frac{m_0}{2}.
\]
We have (omitting $t$-dependence)
\[
\int_{-\infty}^x \partial_y(\rho u)\rd{y} = \rho(x)u(x) = \rho(x)\Big(-\int  \cN'(x-y)\rho(y)\rd{y}  - \cV'(x)\Big),
\]
and 
\[
\begin{split}
-\int  \cN'(x-y)&\rho(y)\rd{y} =  -\int_{-\infty}^\infty \cN'(x-y)\partial_y m(y)\rd{y}\\
= & -\lim_{y\rightarrow\infty} \cN'(x-y)m(y) + \lim_{y\rightarrow-\infty} \cN'(x-y)m(y) - \int_{-\infty}^\infty \cN''(x-y)m(y)\rd{y} \\
= & -\frac{1}{2}\cdot \frac{m_0}{2} + (-\frac{1}{2})\cdot (-\frac{m_0}{2}) + m(x) = m(x) .
\end{split}
\]
Therefore, by integrating \eqref{eq1D} in $x$, we see that $m(t,x)$ satisfies
\[
\partial_t m + (m(x)-\cV'(x))\partial_x m = 0.
\]

For fixed $t$, since $m(t,x)$ is an increasing function in $x$, one can define $X(t,m)$ as its inverse function, except a countable set of values of $m$. Then $X(t,m)$, for almost all $m\in (-m_0/2,m_0/2)$, satisfies an ODE
\begin{equation}\label{Xode}
\frac{\rd}{\rd{t}}X(t,m) = m - \cV'(X).
\end{equation}
Therefore, as long as $\cV$ is super-linear:
\[
\lim_{x\rightarrow\infty} \cV'(x) = \infty, \quad \lim_{x\rightarrow-\infty} \cV'(x) = -\infty.
\]
\eqref{Xode} drives $X(t,m)$ to the equilibrium point $x$ with $\cV'(x) = m$, which lies in the same basin of attraction as the initial data $X_{in}(m)$. If $\cV$ is strictly convex, then there is a unique $x$ with $\cV'(x) = m$; otherwise there may be more than one $x$. Therefore we conclude:

\begin{proposition}

If $\cV$ is super-linear, then the solution to \eqref{eq1D} with compactly supported initial data converges to a steady state as $t\rightarrow\infty$, in the sense that $\lim_{t\rightarrow\infty} X(t,m) = X_\infty(m)$ for almost all $m\in(-m_0/2,m_0/2)$, for some $X_\infty(m)$ with $\cV'(X_\infty(m))=m$. 

If in addition, $\cV$ is strictly convex, then the steady state is unique for each fixed $m_0$; if $\cV''(x)\ge a>0,\,\forall x$, then the convergence rate of the limit $\lim_{t\rightarrow\infty} X(t,m) = X_\infty(m)$ is exponential, being uniform in $m$. 

If $\cV$ is not convex, then the steady state may fail to be unique.

\end{proposition}


\subsection{Steady states must have compact support}\label{sec:must_be_compact}
\begin{proposition}\label{prop_cpt}
Let $d\ge 2$, and $\cV$ be a radial potential satisfying \rev{$\Delta \cV(\bx)\ge 0,\,\forall \bx\in\mathbb{R}^d$, }$\|\Delta \cV\|_{L^\infty} < \infty$ and  the condition:
\begin{equation}\label{prop_cpt_1}
V'(r) \ge c_\cV r^{-\frac{d-1}{d+1}} ,\quad \forall r\ge R_0,
\end{equation}
for some $R_0>0$, where $ c_\cV > 0$ . Then any steady state of \eqref{eq} has compact support.
\end{proposition}

\begin{proof}
Let $\rho=\Delta \cV \chi_{\supp\rho}$ be a steady state, and take $R>0$. We aim to prove that when $R$ is large enough, then $\supp\rho \cap \{|\bx|=R\}=\emptyset$. In the rest of the proof, we denote
\begin{equation}\label{epsilonR}
\epsilon_R = \int_{|\bx|>R} \rho(\bx)\rd{\bx},\quad \text{satisfying }\lim_{R\rightarrow\infty} \epsilon_R = 0.
\end{equation}

Suppose the contrary, then we take $\bx \in\supp\rho \cap \{|\bx|=R\}$, and we may assume $\bx = (R,0,\dots,0)^T$ without loss of generality. The steady state equation \eqref{rhoinfty} implies
\[
-\int \nabla\cN(\bx-\by)\rho(\by)\rd{\by} - \nabla \cV(\bx) = 0.
\]
Taking inner product with $\bx$ gives
\begin{equation}\label{force30}
-\int \bx\cdot\nabla\cN(\bx-\by)\rho(\by)\rd{\by} - V'(R)R = 0.
\end{equation}

We aim to show that the LHS is negative which leads to a contradiction. We first write
\begin{equation}\label{force3}\begin{split}
 -\int \bx\cdot\nabla\cN(\bx-\by)\rho(\by)\rd{\by} & =   c\int \frac{\bx\cdot (\bx-\by)}{|\bx-\by|^d} \rho(\by)\rd{\by} \le c\int_{y_1\le R}  \frac{\bx\cdot(\bx-\by)}{|\bx-\by|^d} \rho(\by)\rd{\by} \\
& \le  -\int_{|\by|\le R} \bx\cdot\nabla\cN(\bx-\by) \rho(\by)\rd{\by} \\
 & \quad + c\int_{R-\delta \le y_1 \le R}  \frac{\bx\cdot(\bx-\by)}{|\bx-\by|^d} \rho(\by)\rd{\by} \\
  & \quad + c\int_{S}  \frac{\bx\cdot(\bx-\by)}{|\bx-\by|^d} \rho(\by)\rd{\by} ,
\end{split}\end{equation}
where $y_1$ denotes the first component of $\by$,  $\delta>0$ is small, to be determined, and
\begin{equation}
S := \{\by: y_1\le R\} \backslash \Big( \BR \cup \{\by:R-\delta \le y_1 \le R\} \Big).
\end{equation}
Now we estimate the three terms on the RHS of \eqref{force3} separately:

{\bf The first term} (combined with the term $V'(R)R$ in \eqref{force30}). Similar to STEP 3 of the proof Theorem \ref{thm_Vequi}, we \rev{use the assumption $\Delta \cV\ge 0$ and} write
\[
\begin{split}
-\int_{|\by|\le R} \bx\cdot\nabla\cN(\bx-\by) \rho(\by)\rd{\by} - V'(R)R = & \int_{\BR\backslash \supp\rho} \bx\cdot\nabla\cN(\bx-\by) \Delta \cV(\by)\rd{\by} \\ \le & -\frac{c}{R^{d-2}}\int_{\BR\backslash \supp\rho}  \Delta \cV(\by)\rd{\by}.
\end{split}
\]
Notice that by the assumption \eqref{prop_cpt_1},
\[
\int_{|\by|\le R}  \Delta \cV(\by)\rd{\by} = \int_{|\by|=R}\nabla \cV(\by)\cdot \vec{n} \rd{S(\by)} = cR^{d-1}V'(R) \ge cR^{d-1-\frac{d-1}{d+1}},
\]
for $R$ sufficiently large, and
\[
 \int_{\supp\rho}\Delta \cV(\by)\rd{\by} = m_0.
\]
Therefore, since $d-1-\frac{d-1}{d+1}>0$, we get
\[
\begin{split}
-\int_{|\by|\le R} \bx\cdot\nabla\cN(\bx-\by) \rho(\by)\rd{\by} - V'(R)R  \le & -\frac{c}{R^{d-2}}\cdot R^{d-1-\frac{d-1}{d+1}} = -cR^{\frac{2}{d+1}}.
\end{split}
\]

{\bf The second term}. One can show that for fixed $y_1<R$, writing $\by = (y_1,\by'),\,\by'\in\mathbb{R}^{d-1}$,
\[
\frac{\bx}{|\bx|}\cdot\int_{\mathbb{R}^{d-1}}\frac{(\bx-\by)}{|\bx-\by|^d} \rd{\by'} = C,
\]
is independent of $y_1$. In fact,
\[
\begin{split}
\frac{\bx}{|\bx|}\cdot\int_{\mathbb{R}^{d-1}}\frac{(\bx-\by)}{|\bx-\by|^d} \rd{\by'} & = \int_{\mathbb{R}^{d-1}} \frac{R-y_1}{((R-y_1)^2+(\by')^2)^{d/2}}\rd{\by'}\\
 & = \int_{\mathbb{R}^{d-1}} \frac{1}{(1+(\by')^2)^{d/2}}\rd{\by'} = C.
 \end{split}
\]

 Therefore, using the assumption $\|\Delta \cV\|_{L^\infty} < \infty$, we get
\[
\int_{R-\delta \le y_1 \le R}  \frac{\bx\cdot(\bx-\by)}{|\bx-\by|^d} \rho(\by)\rd{\by} \le CR\int_{R-\delta \le y_1 \le R} \frac{\bx}{|\bx|}\cdot\int_{\mathbb{R}^{d-1}} \frac{(\bx-\by)}{|\bx-\by|^d}\rd{\by'}\rd{y_1} \le C\delta R.
\]

{\bf The third term}.
We claim that
\begin{equation}\label{claimS}
|\bx-\by| \ge \sqrt{\delta R},\quad \forall \by\in S.
\end{equation}
For those $\by$ with $y_1<0$, this is clear because $|\bx-\by| \ge R$ in this case. For those $\by=(y_1,\by')$ with $y_1\ge 0$, notice that
\[
|\bx-\by|^2 = (R-y_1)^2 + |\by'|^2 \ge |\by'|^2 = |\by|^2 - y_1^2.
\]
By the definition of $S$, we have $|\by|^2 \ge R^2$ and $y_1^2 \le (R-\delta)^2$. Therefore
\[
|\bx-\by|^2 \ge R^2 - (R-\delta)^2 = 2\delta R - \delta^2 \ge \delta R,
\]
using the smallness of $\delta$. This proves the claim.

Using \eqref{claimS}, we get
\[
\frac{|\bx\cdot(\bx-\by)|}{|\bx-\by|^d} \le R\cdot \frac{1}{|\bx-\by|^{d-1}} \le R\cdot (\delta R)^{-(d-1)/2} = \delta^{-(d-1)/2}R^{-(d-3)/2},
\]
which together with the assumption $\|\Delta \cV\|_{L^\infty} < \infty$, gives the estimate
\[
\int_{S}  \frac{\bx\cdot(\bx-\by)}{|\bx-\by|^d} \rho(\by)\rd{\by} \le C\epsilon_R \delta^{-(d-1)/2}R^{-(d-3)/2},
\]
using the fact that $S\cap \BR = \emptyset$.

Now we take
\[
\delta = \epsilon_R^{2/(d+1)} R^{-(d-1)/(d+1)},
\]
to equate the second and third terms, and finally obtain the estimate
\[
\begin{split}
0 \le & -\int_{|\by|\le R} \bx\cdot\nabla\cN(\bx-\by) \rho(\by)\rd{\by}- V'(R)R + c\int_{R-\delta \le y_1 \le R}  \frac{\bx\cdot(\bx-\by)}{|\bx-\by|^d} \rho(\by)\rd{\by} \\
& + c\int_{S}  \frac{\bx\cdot(\bx-\by)}{|\bx-\by|^d} \rho(\by)\rd{\by} \le   -cR^{2/(d+1)} + C\epsilon_R^{2/(d+1)} R^{2/(d+1)}.
\end{split}
\]
This gives the desired contradiction for large enough $R$, in view of \eqref{epsilonR}.

\end{proof}

\end{document}